\documentclass[oneside]{article}
\usepackage[margin=1in]{geometry}
\usepackage{verbatim}
\usepackage{graphicx}
\usepackage{amssymb}
\usepackage{amsthm}
\usepackage{latexsym}
\usepackage{amsmath}

\usepackage{amsfonts}
\usepackage[colorlinks=true]{hyperref}

\usepackage{tikz}
\usetikzlibrary{matrix}
\usetikzlibrary{positioning}
\usetikzlibrary{intersections} % name intersections
\usetikzlibrary{calc} % coordinate calculations
\numberwithin{equation}{section}
\numberwithin{figure}{section}

\theoremstyle{plain}
\newtheorem{theorem}{Theorem}[section]

\newtheorem{lemma}[theorem]{Lemma}
\newtheorem{lem}[theorem]{Lemma}
\newtheorem{proposition}[theorem]{Proposition}

\newtheorem{corollary}[theorem]{Corollary}
\newtheorem{cor}[theorem]{Corollary}

\theoremstyle{definition}
\newtheorem{observation}[theorem]{Observation}
\newtheorem{remark}[theorem]{Remark}
\newtheorem{rmk}[theorem]{Remark}
\newtheorem{example}[theorem]{Example}
\newtheorem{definition}[theorem]{Definition}
\newtheorem{definitions}[theorem]{Definitions}

\makeatletter
\g@addto@macro\@floatboxreset\centering
\let\@fnsymbol\@alph
\makeatother

\newcommand{\Z}{\mathbf{Z}}
\newcommand{\R}{\mathbf{R}}
\newcommand{\Q}{\mathbf{Q}}
\newcommand{\N}{\mathbf{N}}
\newcommand{\bs}{\smallsetminus}
\newcommand{\abs}[1]{\left\lvert#1\right\rvert}

\newcommand{\ceil}[1]{\left\lceil#1\right\rceil}

\newcommand{\gen}[1]{{\left\langle#1\right\rangle}}
\newcommand{\denom}[1]{\ulcorner#1\lrcorner}
\newcommand{\segment}{\overline{pq}}
\newcommand{\fig}[1]{Figure~\ref{fig:#1}}
\newcommand{\s}{\sigma}
\renewcommand{\l}{\ell}
\renewcommand{\i}{{}^{-1}}

\newcommand{\latin}[1]{\textsl{#1}}

\newcommand{\defn}[1]{\textbf{#1}}

\newcounter{note}

\newcommand{\tumble}[1]{#1}
\renewcommand{\r}{\tumble{r}}
\renewcommand{\u}{\tumble{u}}

\title{Geodesic trajectories on regular polyhedra}

\begin{document}

\author{Diana Davis\thanks{Mathematics Department, Northwestern University, Evanston, IL 60208, USA; \texttt{diana@math.northwestern.edu}} \and
Victor Dods\thanks{Department of Mathematics, University of California Santa Cruz, Santa Cruz, CA 95064, USA; \texttt{vdods@ucsc.edu}; partially supported by the 2011--2012 ARCS Fellowship.} \and
Cynthia Traub\thanks{Department of Mathematics $\&$ Statistics, Southern Illinois University Edwardsville, Edwardsville, IL 62026, USA; \texttt{cytraub@siue.edu}} \and
Jed Yang\thanks{School of Mathematics, University of Minnesota, Minneapolis, MN 55455, USA; \thinspace \texttt{jedyang@umn.edu}; partially supported by NSF GRFP grant DGE-0707424 and NSF RTG grant DMS-1148634.}
}
\maketitle

\begin{abstract}
Consider all geodesics between two given points on a polyhedron.
On the regular tetrahedron,
we describe all the geodesics from a vertex to a point, which could be another vertex.
Using the Stern--Brocot tree to explore the recursive structure of geodesics between vertices on a cube,
we prove, in some precise sense,
that there are twice 
as many geodesics
between certain pairs of vertices than other pairs.
We also obtain the fact that there are no geodesics that start and end at the same vertex on the regular tetrahedron or the cube.
\end{abstract}


\section{Introduction}
Geodesics on surfaces are curves that are locally shortest.
We consider the following problem:
Given a point $p$ on a polyhedron and a vertex~$v$, what are the geodesics from $p$ to~$v$?
In this paper, we focus on the cube, and in particular on vertex-to-vertex geodesics,
that is, where $p$ is also a vertex.
Even though geodesics on convex polyhedra do not pass through vertices,
it is possible for a geodesic to start and end at the same vertex (a ``geodesic loop'').
However, it turns out that this does not happen for the regular tetrahedron and the cube (see Corollaries~\ref{cor:tet-noloop} and \ref{cor:cube-noloop}, respectively),
so we consider $p$ distinct from~$v$.

\subsection{Previous results} 
Fuchs and Fuchs \cite{Fuchs} studied closed geodesics on regular polyhedra.
They give the (previously known, simple) result for the regular tetrahedron and also describe all closed geodesics on the cube and regular octahedron.

For the regular tetrahedron (\cite[\S3]{Fuchs}), they show that every geodesic is non-self-intersecting and give a simple characterization of every closed geodesic.
Since it does not intersect itself, every closed geodesic cuts the surface of the tetrahedron into two pieces.
For the cube (\cite[\S4]{Fuchs}), they show that, up to symmetry and translation, there are exactly three non-self-intersecting closed geodesics, two of which are planar.
They also show that for a self-intersecting closed geodesic, all the self intersections are perpendicular.
They exhibit all three non-self-intersecting geodesics and show several examples of longer geodesics with many self intersections.

More generally, over the past century many authors have studied straight paths on polyhedra.
Alexandrov showed in 1950 that a shortest path never passes through a vertex of positive curvature~\cite{Alexandrov}.
Similarly, Sharir and Schorr showed that on a convex polyhedron, shortest paths do not pass through vertices~\cite{SS},
although they may do so on a non-convex polyhedron~\cite{Mitchell}.
Moreover, shortest paths never self-intersect or intersect the same face more than once.
Locally, closed geodesics are shortest paths and therefore inherit some of these properties.
In particular, a closed geodesic cannot pass through a vertex.
In contrast, a geodesic loop can start and end at the same vertex,
but its ends cannot join smoothly to form a locally shortest path at the vertex.
%This means that the closed geodesics described in \cite{Fuchs} summarized above are not ``geodesics'' in the sense of being a shortest path on the polyhedron:
%they are geodesics in the sense of being locally straight lines, with a smooth junction at the beginning and ending point.
These results are discussed in Demaine and O'Rourke's excellent comprehensive book (\cite[\S24]{DO}).

\subsection{Results presented in this paper}

In Section~\ref{sec:tetrahedron}, we characterize all of the geodesics from any point to any vertex on the regular tetrahedron. 

{\bf Theorem~\ref{thm:ptet}:} We describe the directions from a given point $p$ in which a geodesic will end at a given vertex~$v$.

{\bf Corollary~\ref{cor:tet}:} Given a pair of (necessarily distinct) vertices $v_0$ and~$v$, we give a complete description of the directions of geodesics from $v_0$ to~$v$.

In Section~\ref{sec:cubeintro}, we introduce our conventions and give basic results about geodesics on the cube.

In Section~\ref{sec:v2v}, we develop the beautiful recursive structure underlying vertex-to-vertex paths on the cube, based on the Stern--Brocot tree. We  use this structure to derive several results:

{\bf Theorem~\ref{thm:proportions}:} There are twice as many geodesics to the cube vertex at greatest distance from the starting vertex
as there are to each of the three vertices that are diagonally opposite the starting vertex along a common face,
and $1.5$ times as many as to the three adjacent vertices.
(The notion of ``twice as many'' is made rigorous below.)

{\bf Corollary~\ref{ways}:} We count the number of geodesics to a given vertex, depending on the ``depth'' in the Stern--Brocot tree.

In Section~\ref{sec:f2v}, we consider geodesics starting from a general point in a face by
associating them with line segments from a point in the unit square to a lattice point in~$\R^2$.

{\bf Lemma~\ref{latticevisibility}:} We characterize the lattice points that are reachable (``visible'') from a starting point~$p$ in a face.

{\bf Proposition~\ref{p2vsequence}:} We give an algorithm for determining the ``tumble sequence'' to a given lattice point.

\section{Basic definitions} \label{sec:basics}
Given a polyhedron,
a \defn{geodesic} is a locally shortest curve on its surface.

Our goal is:
Given a polyhedron~$S$, a distinguished vertex~$v$, and a distinguished point~$p$ (which may or may not be a vertex),
determine all of the geodesics from $p$ to~$v$. 
We consider a ray starting at~$p$, following the surface of~$S$, possibly wrapping around $S$ many times, before finally arriving at~$v$.
To do this, we \defn{unfold} the faces of~$S$ in the following way.
For concreteness, suppose that the face containing $p$ is on the $xy$-plane.
When the ray hits an edge~$e$,
we \defn{tumble} $S$ by applying the unique orientation-preserving isometry that fixes edge~$e$
and places the adjacent face of $S$ on the $xy$-plane on the other side of~$e$.% 
\footnote{Think of this as rolling a polyhedral die on a tabletop without slipping.}
After the tumble, the geodesic continues on this new face.
These two segments in the $xy$-plane form a straight line segment by definition.
We continue in like manner until the geodesic hits a vertex.
In such a way, a geodesic on $S$ naturally corresponds to a straight line segment in the $xy$-plane,
contained in a strip formed by copying the faces of $S$ to the $xy$-plane as $S$ tumbles.
Note that the unfolding depends on the geodesic being considered.

In this paper, we study special cases of this polyhedron problem:
the regular tetrahedron and the cube.
These polyhedra are especially elegant because their faces tile the plane.
When either of these polyhedra is tumbled in all possible ways,
the points in the $xy$-plane touched by the vertices form a lattice~$\Lambda$.
We note that $\Lambda$ is the equilateral triangle lattice for the tetrahedron and the square lattice for the cube.
 
\begin{definition} \label{def:visible}
Given a point $p\in\R^2$,
a lattice point $q\in\Lambda$ distinct from $p$ is \defn{visible (from $p$)}
if the interior of the segment $\segment$ from $p$ to $q$ does not contain any lattice points of~$\Lambda$ (see \fig{orchard}).
\end{definition}

\begin{lemma} \label{lem:visible}
Let $S$ be the regular tetrahedron or the cube and $\Lambda$ the corresponding lattice.
A segment $\segment$ on the $xy$-plane corresponds to a geodesic of $S$ ending at a vertex if and only if $q\in\Lambda$ is a visible lattice point.
%and the segment contains no other lattice points in its interior.
%The endpoint of a geodesic from point $P$ to vertex $V$ is a lattice point of~$\Lambda$.
%The lattice points of $\Lambda$ are precisely the points that are endpoints of geodesics from $P$ to some vertex.
\end{lemma}
\begin{proof}
We reconstruct the corresponding geodesic given the segment on the plane:
each time the segment leaves the current face through an edge,
we tumble $S$ across that edge.
The added vertex of this new face is of course a lattice point.
If $q$ is not a lattice point, then the corresponding geodesic does not end at a vertex.
On the other hand, if $q$ is a lattice point that is not visible,
then the corresponding geodesic passes through a vertex, a contradiction.
\end{proof}
We identify the geodesics on~$S$ with the corresponding segments in the $xy$-plane (see \fig{tumblepathex}).

%Victor Dods, Diana Davis, Jed Yang - 2013.08.20 - Overview of results-thus-far
%on geodesics between vertices of a cube.
%
%Goal: Label the vertices of an axially aligned unit cube (hereafter
%referred to as simply a cube) in the following way, and quantify all
%geodesics%
%\footnote{Really only consider geodesics up to a certain length.%
%} on the surface of the cube joining vertex 0 to a given vertex.

{
\section{The regular tetrahedron} \label{sec:tetrahedron}
\renewcommand{\u}{\mathbf{u}}
\renewcommand{\v}{\mathbf{v}}
\renewcommand{\a}{\bar a}
\renewcommand{\b}{\bar b}

Let $T$ be the regular tetrahedron whose edges are all of unit length.
Given a point $p$ on the surface of $T$ and a vertex $v$ of~$T$,
we wish to characterize all geodesics from $p$ to~$v$.

Label the vertices of $T$ by the elements of $(\Z/2)^2$,
and think of $T:(\Z/2)^2\to\R^3$ as the location of its four vertices.
For concreteness, suppose $T$ is initially placed with $T(0,0)=(0,0,0)$ at the origin,
$T(1,1)=(1,0,0)$ on the positive $x$-axis,
$T(1,0)=\u=\big(\frac12,\frac{\sqrt3}2,0\big)$ on the first quadrant of the $xy$-plane,
and $T(0,1)=\big(\frac 12, \frac 1{2\sqrt{3}}, \sqrt{2/3}\big)$. %with positive $z$-coordinate 
Moreover, suppose $p$ is on the $xy$-plane.

A geodesic on $T$ starting at $p$ and ending at a vertex is uniquely determined by its initial angle~$\alpha$ above the horizontal:
it corresponds to a ray in the $xy$-plane emanating from $p$ with slope $\tan\alpha$.
The question becomes:
Given a segment in the $xy$-plane,
determine whether it corresponds to a geodesic and, if so, the vertex label of its endpoint.

This follows easily from two simple lemmas.
Let $\Lambda$ be the lattice generated by
\[ \textstyle \u=\big(\frac12,\frac{\sqrt3}2,0\big)\text{\quad and\quad}\v=\big(\frac12,-\frac{\sqrt3}2,0\big). \]
A lattice point $q\in\Lambda$ can be uniquely written as $q=a\u+b\v$ for some $a,b\in\Z$.
Label $q$ with $(\a,\b)$,
%Label a lattice point $a\u+b\v$ by $a'+2b'$, 
where $\a,\b\in\Z/2$ are the images of $a$ and $b$, respectively, under the canonical surjection $\Z\twoheadrightarrow\Z/2$.
This gives a labeling map $\chi:\Lambda\to(\Z/2)^2$.

\begin{lemma}[Labeling] \label{lem:label}
If a segment $\segment$ on the $xy$-plane corresponds to a geodesic of $T$ ending at a vertex~$v$,
then $v$ is the label of the endpoint $q$, that is, $\chi(q)=v$.
%If a line segment from $P$ to a lattice point $\lambda$ does not pass through any other lattice point,
%then on the plane ends at a lattice point $\lambda$, the geodesic on $T$ ends at the vertex with the same label as~$\lambda$.
%The labels of the lattice points are the labels of the vertices.
\end{lemma}
\begin{proof}
We need only check that the labels of the lattice points are locally consistent.
First, notice that the lattice points form unit equilateral triangles.
Pick two such triangles $\Delta_1$ and $\Delta_2$ that share an edge~$e$.
It is easy to check that the four lattice points have distinct labels.
As such, if $T$ rests on top of~$\Delta_1$ such that its vertices are consistent with the labels of the lattice points,
then tumbling $T$ across $e$ will cause $T$ to sit on top of $\Delta_2$, again with consistent labels.
Since this works for any pair of adjacent triangles, we are done.
\end{proof}

Finally, we derive an algebraic criterion for the visibility of a lattice point.
Note that the starting point $p$ can be uniquely written as $p=x\u+y\v$ for some $x,y\in[0,1)$.
%\dnote{with $x+y<1$?}
%\jnote{Nope, $(1,0,0)=1\u+1\v$.}
%\dnote{then maybe $x,y\in[0,1]$?}
By rotating and relabeling $T$, without loss of generality, we shall assume that $\alpha\in(-\pi/3,\pi/3]$.
%\dnote{why not $\alpha\in[0,\pi/3]$?}
%\jnote{We have a $3$-fold symmetry, so we should get $2\pi/3$ total angle at the end (at least).
%Let $p=(0.9,0.1)$ and draw the ray towards $\u+(0.1,0)$.
%This ray is angled $2\pi/3-\epsilon>\pi/3$.
%I'd rather rotate this and have it come out of the bottom edge in the range I specified.
%This makes the arithmetic easier since $a>x$ and $b\geq y$ describes this region.}
That is, only consider lattice points $q=a\u+b\v$ where $a>x$ and $b\geq y$.
A portion of the labelled lattice is shown in Figure~\ref{triangles}.

\begin{figure}[hbtp]
\begin{tikzpicture}[scale=1.5]
\pgfmathsetmacro{\xcoord}{cos(60)}
\pgfmathsetmacro{\ycoord}{sin(60)}
%\pgftransformcm{1}{0}{\xcoord}{\ycoord}{\pgfpointorigin}
\coordinate (O) at (0,0);
\coordinate (u) at (\xcoord,\ycoord);
\coordinate (v) at (\xcoord,-\ycoord);

\clip (-1,-2) rectangle (7,2);

\fill [opacity=.3] (O) -- (u) -- ($(u)+(v)$);

\draw[-latex,ultra thick] (O) -- node[above left] {$\u$} (u) ;
\draw[-latex,ultra thick] (O) -- node[below left] {$\v$} (v) ;

\foreach \y in {0,...,6}
{
\foreach \x in {0,...,6}
   {
      \fill[black,opacity=.7] ($\x*(u) + \y*(v)$) circle (1.3pt);
      \path ($\x*(u) + \y*(v)$) node[right] {$(\pgfmathparse{int(mod(\x,2))} \pgfmathresult, \pgfmathparse{int(mod(\y,2))} \pgfmathresult)$};
   }
}
\end{tikzpicture}
\caption{The labeling of the triangular lattice~$\Lambda$. The starting point is in the grey region. \label{triangles}}
\end{figure}
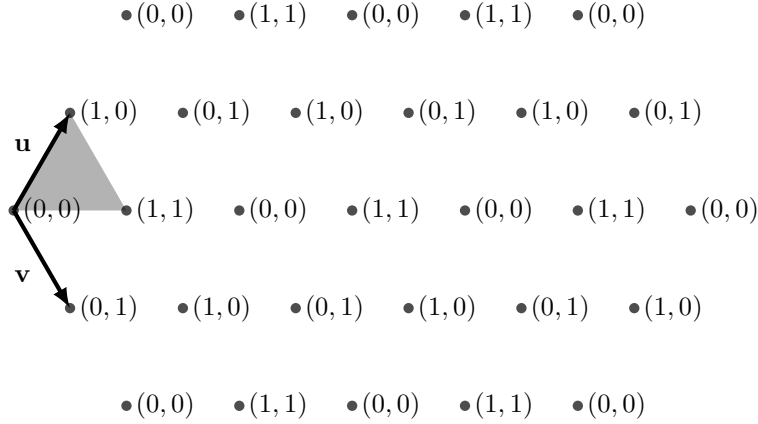  

\begin{definition} \label{denominator}
For $x\in\R$, define the \defn{denominator} $\denom x$ as 

\[ \denom x=\left\{\begin{array}{ll}0 & \text{if $x\notin\Q$,} \\ \frac1d & \text{if $x=\frac cd$, $c\in\Z$, $d\in\Z_+$, and $\gcd(c,d)=1$.}\end{array}\right. \]
\end{definition}

Note that for rational input~$x$, the value of $\denom x$ is the denominator of the lowest terms fraction representing~$x$.

\begin{lemma}[Visibility] \label{lem:denom}
For $a>x$ and $b\geq y$, the lattice point $q=a\u+b\v$ is visible from $p=x\u+y\v$ if and only if $\denom{\frac{b-y}{a-x}}\leq\frac1a$.
\end{lemma}
\begin{proof}
If $q$ is not visible, there exists $r\u+s\v$ that lies on the segment from $p=x\u+y\v$ to $q=a\u+b\v$.
That means $\frac{b-y}{a-x}=\frac{b-s}{a-r}\in\Q$.
Let $d=a-r$; since $0<r<a$, we have $0<d<a$.
Thus $\denom{\frac{b-y}{a-x}}\geq\frac1d>\frac1a$.
Conversely, suppose $\denom{\frac{b-y}{a-x}}=\frac1d>\frac1a$.
Let $c=\frac{b-y}{a-x}\cdot d\in\Z$.
%Then $r=a-d\in\Z_+$, and $s=b-c$.
Then $(a-d)\u+(b-c)\v$ is a lattice point that lies on the segment~$\segment$.
%We have that $r\u+s\v$ lies on the segment.
\end{proof}

\begin{corollary} \label{cor:latvis}
For $a,b\geq0$, a lattice point $q=a\u+b\v$ is visible from the origin if and only if $\gcd(a,b)=1$.
\end{corollary}

\begin{remark}
If we start at a point $p=x\u+y\v$ such that one of $x$ and $y$ is rational and the other is irrational, all lattice points are visible,
because $\frac{b-y}{a-x}$ is irrational for all $a,b\in\Z$.
\end{remark}

%\begin{proof}
%If $(x,y)=(0,0)$, the condition above is saying that the fraction $\frac ba$ cannot be further reduced.
%\end{proof}

We obtain a description of point-to-vertex geodesics by using the lemmas above.

\begin{theorem} \label{thm:ptet}
Let $p=x\u+y\v$ for some $x,y\in[0,1)$.
The vectors corresponding to geodesics from $p$ to vertex $(a',b')$ are
$$\left\{ (a-x)\u+(b-y)\v    \mid \text{$a\equiv a',\ b\equiv b'\pmod2$ and $\denom{\tfrac{b-y}{a-x}}\leq\tfrac1a$} \right\}$$
with corresponding angles $\arctan\tfrac{\sqrt{3}(a-b-x+y)}{(a+b-x-y)}$.
\end{theorem}

Theorem~\ref{thm:ptet} may not be an entirely satisfying classification, because it relies on our non-standard denominator function.
However, this theorem, along with Corollary~\ref{cor:latvis}, gives a complete description of vertex-to-vertex geodesics:

\begin{corollary} \label{cor:tet}
%Let the starting point $P$ of the geodesics be at vertex $0$.
The starting angles of geodesics from vertex $(0,0)$ to vertex $(a',b')$ are
$$\left\{\arctan\tfrac{\sqrt3(a-b)}{a+b}\mid a\equiv a',\ b\equiv b'\pmod2,\ \gcd(a,b)=1\right\}.$$
\end{corollary}

To end at vertex $(0,0)$, the numbers $a$ and $b$ must both be even, contradicting the $\gcd$ condition.
\begin{corollary} \label{cor:tet-noloop}
The regular tetrahedron admits no geodesic loop at a vertex,
\latin{i.e.}, a geodesic that starts and end at the same vertex.
\end{corollary}
It can be seen in Figure~\ref{triangles} that every lattice point labelled $(0,0)$ is not visible from the origin.
}

%\newpage
\section{Introduction to the cube} \label{sec:cubeintro}

We define basic conventions for the cube and make some simple observations about geodesics on the cube, in preparation for our more detailed study in Section~$\ref{sec:v2v}$.

\begin{definitions} \label{basicdefs}
The \defn{starting position} of the cube has opposite vertices at $(0,0,0)$ and $(1,1,1)$.
We label the vertices $0 = (0,0,0)$, $1 = (1,0,0)$, $2 = (0,1,0)$, $3 =(1,1,0)$, $4 = (0,0,1)$, $5 = (1,0,1)$, $6 = (0,1,1)$, $7 = (1,1,1)$ (see Figure~\ref{vertexlabels}).
The vertex coordinates correspond to their labels in binary.

\begin{figure}[hbtp] 
\hfill{}\begin{tikzpicture}
  [back line/.style={densely dotted}, 
   cross line/.style={preaction={draw=white, -, line width=6pt}}
  ] 
  \matrix (m) 
    [matrix of math nodes, 
     row sep=3em, 
     column sep=3em, 
     text height=1.5ex, 
     text depth=0.25ex] 
  { 
    & 6 & & 7 \\ 
    4 & & 5 \\ 
    & 2 & & 3 \\ 
    0 & & 1 \\ 
  }; 
  \path
    (m-1-2) edge (m-1-4) edge (m-2-1) edge [back line] (m-3-2) 
    (m-1-4) edge (m-3-4) edge (m-2-3) 
    (m-2-1) edge [cross line] (m-2-3) edge (m-4-1) 
    (m-3-2) edge [back line] (m-3-4) edge [back line] (m-4-1) 
    (m-4-1) edge (m-4-3) 
    (m-3-4) edge (m-4-3) 
    (m-2-3) edge [cross line] (m-4-3)
  ;
\end{tikzpicture}\hfill{} \begin{tikzpicture}
  \node (0)               {$0$};
  \node (1)  [right=of 0] {$1$};
  \node (2)  [above=of 0] {$2$};
  \node (3)  [above=of 1] {$3$};
  \node (6)  [above=of 2] {$6$};
  \node (7u) [above=of 3] {$7$};
  \node (5)  [right=of 1] {$5$};
  \node (7r) [right=of 3] {$7$};
  \path (0) edge (1) edge (2);
  \path (3) edge (1) edge (7r) edge (7u) edge (2);
  \path (6) edge (2) edge (7u);
  \path (5) edge (1) edge (7r);
\end{tikzpicture}
\hfill{}\begin{tikzpicture}
  [back line/.style={densely dotted}, 
   cross line/.style={preaction={draw=white, -, line width=6pt}}
  ] 
  \matrix (m) 
    [matrix of math nodes, 
     row sep=3em, 
     column sep=3em, 
     text height=1.5ex, 
     text depth=0.25ex] 
  { 
    & 3 & & 7 \\
    2 & & 6 \\
    & 1 & & 5 \\
    0 & & 4 \\
  }; 
  \path
    (m-1-2) edge (m-1-4) edge (m-2-1) edge [back line] (m-3-2) 
    (m-1-4) edge (m-3-4) edge (m-2-3) 
    (m-2-1) edge [cross line] (m-2-3) edge (m-4-1) 
    (m-3-2) edge [back line] (m-3-4) edge [back line] (m-4-1) 
    (m-4-1) edge (m-4-3) 
    (m-3-4) edge (m-4-3) 
    (m-2-3) edge [cross line] (m-4-3)
  ;
\end{tikzpicture}\hfill{}
\par
\begin{quote}\caption{The vertex labels of the cube;
an unfolding of the vertices onto the plane;
the cube orientation $0312$. \label{vertexlabels}} \end{quote}
\end{figure}
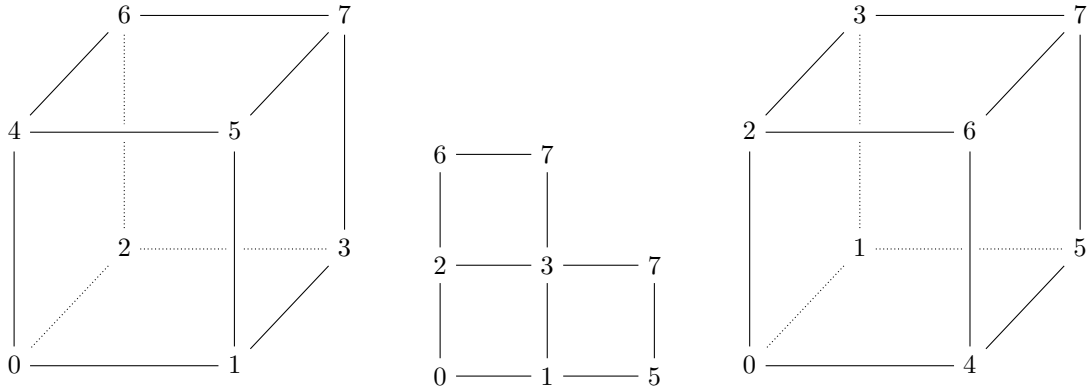

The letter $\r$ represents a \defn{right} roll,
which from the starting position puts vertices $1,3,5$ and $7$ on the $xy$-plane.
The letter $\u$ represents an \defn{up} roll,
which from the starting position puts vertices $2,3,6$ and $7$ on the $xy$-plane.
The direction of rolls are described as seen from the positive $z$-direction looking down at the $xy$-plane.
Hence a right roll is in the positive $x$-direction and an up roll is in the positive $y$-direction.

A \defn{tumble sequence} $\tau$ is a word on the alphabet $\{\r,\u\}$ representing cube rolls.
We read from left to right, so $\tumble{rrurru}$ represents rolling the cube to the right twice, then up once, then to the right twice, and then up once.
We associate $\tau$ with the lattice point $(a,b)$ touched by the upper-right corner of the bottom face of the cube at the end of the tumble,
which we assume to be in the first quadrant by symmetry.

A \defn{tumble path} is the set of squares touched by the faces of the cube as it rolls out the tumble sequence.

We label the four diagonals of the cube $0,1,2,3$ based on which bottom vertex the diagonal touches.
An orientation of the cube is uniquely determined by the permutation of the four diagonals,
so we label each cube orientation with an element in~$S_4$, written in one-line notation.
The cube orientation $\sigma_0=0123$ is the identity,
the starting position of the cube.
The orientation $0312$ in means that diagonal~$0$ is fixed,
diagonal~$3$ is in the starting position of diagonal~$1$,
diagonal~$1$ is in the starting position of diagonal~$2$,
and diagonal~$2$ is in the starting position of diagonal~$3$ (see Figure~\ref{vertexlabels}).

%What we know: $R = \gen{ r, u \mid r^4 = u^4 = (r^2 u)^2 = (u r)^3 = 1 }$.
%Label the vertices of the cube with $0$, $1$, $2$, and $3$ only, where if $z=0$, then label it $2y+x$.  Otherwise, if $z=1$, label it $3-2y-x$.  This is the permuting diagonals idea that I mentioned in my previous email.
The symmetry group of the cube is isomorphic to~$S_4$,
the symmetric group on four letters.
Applying $r$ (respectively, $u$) to a cube corresponds to multiplying its cube orientation by $1230$ (respectively, $2310$).
As such, we identify
%We define 
\begin{eqnarray*}
r &=& 1230\\ %cycle notation: (0123)
u &=& 2310.  %cycle notation: (0213)
\end{eqnarray*}
%where elements of $S_4$ are written in one-line notation.
%Namely, a permutation $\s$ mapping $\{0,1,2,3\}$ to itself is written as $\s(0)\s(1)\s(2)\s(3)$.
This identification induces a canonical monoid homomorphism %from tumble sequences to cube orientations,
that associates to each tumble sequence the cube orientation resulting from applying the tumble sequence to the cube from its starting position.
By an abuse of notation,
a sequence of the letters $r$ and $u$ may represent a tumble sequence (with no cancellations) or the associated cube orientation.
In the latter case, being considered in $S_4$, we may write $r\i$ and $u\i$ for the inverse permutations
or simplify using relations such as $r^4=u^4=(r^2 u)^2=(ur)^3=1$.

Given a line segment between a point $p$ in the unit square $[0,1]^2$, and a lattice point $(a,b)$ in the first quadrant,
such that the segment contains no interior lattice points,
there is a unique corresponding \defn{tumble sequence}~$\tau$:

Construct the square grid $\{\Z\times \R \} \cup \{ \R \times \Z \}$.
When the segment crosses a vertical grid line, record an~$r$; when the segment crosses a horizontal grid line, record a~$u$.
The sequence of the letters $r$ and $u$ is the corresponding tumble sequence~$\tau$.
\end{definitions}
We restrict our attention to tumble sequences that correspond to geodesic trajectories on the cube.

In Section~\ref{sec:v2v}, we take the starting point $p$ to be the origin.
In Section~\ref{sec:f2v}, we consider a general starting point in the unit square.

%\subsection{Generating the tumble sequence for a lattice point $(m,n)$}

\begin{lemma} \label{tumbleseq}
To tumble a cube from $(0,0)$ so that its upper-right corner touches the lattice point $(a,b)$ with $b<a$,
following a geodesic trajectory with no interior lattice points,
the tumble sequence is a word of length $a+b-2$,
consisting entirely of $\r$s except with $\u$s in positions $\ceil{\frac{a}{b} i} + i -1$ for $i = 1,\ldots,b-1$.
In the symmetric case where $b>a$, the roles of $\r$ and $\u$, and of $a$ and $b$, are reversed.
\end{lemma}
\begin{proof}
This is easily derived by the same logic as for drawing straight lines on a computer screen, pixel by pixel.
\end{proof}

\begin{corollary}
Every geodesic tumble sequence is a palindrome.
\end{corollary}
\begin{proof}
This follows from the half-turn symmetry of the axis-parallel rectangle whose opposite vertices are $(0,0)$ and $(a,b)$.
\end{proof}

%\subsection{Determining the end vertex of a tumble sequence}
%Is this vertex-to-vertex only?
%No, a tumble sequence yields a unique end vertex no matter which starting point generated it. The end point is a vertex.
As we roll the cube along the tumble path, we can label the vertices that touch the $xy$-plane (see \fig{tumblepathex}).
%The main question is, given a tumble sequence, how to determine what the upper-right vertex of the cube will be at the end.

\begin{proposition} \label{prop:sumto7}
In each $2$-by-$1$ or $1$-by-$2$ rectangle that makes up the tumble path, opposite vertex labels sum to $7$ (see \fig{tumblepathex}).
\end{proposition}

\begin{proof}
Vertices that appear opposite each other in the $2$-by-$1$ rectangles that make up the tumble path are opposite ends of diagonals of the cube.
Because we used the cube's coordinates to label the vertices, opposite ends of diagonals have $x$-, $y$- and $z$-coordinates that each sum to $1$.
Since each coordinate sums to $1$, the binary expansion of the sum is $111$, or $7$.
\end{proof}

\begin{figure}[hbtp]
\begin{tikzpicture}
  \node (0)               {$\mathbf{0}$};
  \node (1)  [right=of 0] {$\mathbf{1}$};
  \node (2)  [above=of 0] {$\mathbf{2}$};
  \node (3)  [above=of 1] {$\mathbf{3}$};
  \node (6)  [above=of 2] {\phantom{$1$}};
  \node (7u) [above=of 3] {\phantom{$1$}};
  \node (5)  [right=of 1] {$5$};
  \node (7r) [right=of 3] {$7$};
    \node (a)  [right=of 5] {$4$};
    \node (b)  [right=of 7r] {$6$};
     \node (c)  [above=of b] {$2$};
     \node (d)  [above=of 7r] {$3$};
     \node (e)  [right=of c] {$0$};
     \node (f)  [right=of b] {$4$};
     \node (g)  [right=of e] {$1$};
     \node (h)  [right=of f] {$5$};
     \node (i)  [right=of a] {\phantom{$1$}};
     \node (j)  [right=of i] {\phantom{$1$}};
  \path (0) edge (1) edge (2);
  \path (3) edge (1) edge (7r) edge (7u) edge (2);
  \path (6) edge (2) edge (7u) ;
  \path (5) edge (1) edge (7r) ;
  \path (7r) edge (b)  ;
  \path (b) edge (f);
  \path (f) edge (h);
  \path (7u) edge (d);
  \path (d) edge (c)  ;
  \path (c) edge (e);
  \path (e) edge (g);
  \path (g) edge (h);
    \path (d) edge (7r)  ;
  \path (b) edge (c);
  \path (g) edge (h);
  \path (5) edge (a);
    \path (a) edge (b)  ;
  \path (a) edge (i);
  \path (i) edge (j);
  \path (j) edge (h);
  \path (f) edge (i);
  \path (f) edge (e);
  \path (0) edge (g); % this is the diagonal across the picture
\end{tikzpicture}
\begin{quote}\caption{The tumble path for the tumble sequence $rrurr$, corresponding to the diagonal geodesic. \label{fig:tumblepathex}} \end{quote}
\vspace{-2.0em}
\end{figure}
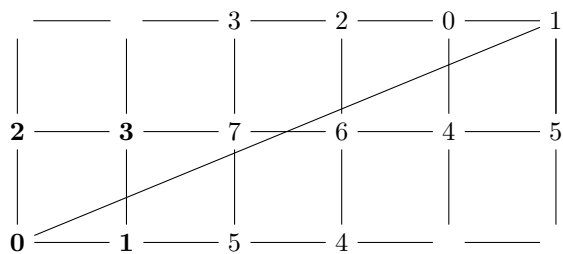

Proposition~\ref{prop:sumto7} gives an easy way to move through a tumble path:
if the vertex labels in \fig{tumblepathex} other than the initial four (in bold) are erased,
they can all be filled back in by using Proposition~\ref{prop:sumto7}.

%\subsection{Shrinking tumble paths \label{shrinkingwords}}

The proof of the following observation is trivial and omitted.

\begin{observation}
If a given tumble word is equivalent to the identity then its mirror reverse is also equivalent to the identity.
\end{observation}

\begin{comment}
\begin{proof}
Suppose $u^k r^l u^m r^n = 1$, for some $k,l,m,n \in \Z^+$.
Let $d$ and $l$ denote \defn{down} and \defn{left} rolls, respectively.
\begin{eqnarray*}
u^k r^l u^m r^n = 1 &\implies& d^k l^l d^m l^n = 1 \hspace{2cm} \text{by symmetry between $d$ and $u$, and $r$ and $l$}\\
&\implies& u^{4-k} r^{4-l} u^{4-m} r^{4-n} = 1 \hspace{2cm} \text{because $d^k=u^{4-k}$ and $l^k=r^{4-k}$}\\
&\implies& u^{4-k} r^{4-l} u^{4-m}  = r^{n} \\
&\implies& u^{4-k} r^{4-l}   = r^{n}u^{m} \\
&\implies& u^{4-k}   = r^{n}u^{m} r^{l} \\
&\implies& 1   = r^{n}u^{m} r^{l}u^{k}.
\end{eqnarray*}
A word of length different from $4$, or beginning and ending with the same letter ($r$ or $u$), can be mirror reversed in the same manner.
\end{proof}
\end{comment}

%\begin{conjecture}
%Every tumble word can be reduced, via the relations $$r^4=u^4=(r^2 u)^2=(u^2 r)^2=(ur)^3=1$$ to the following short list of irreducible words: 
%\end{conjecture}
%$$\textrm{[Victor's \ data]}$$
%\jnote{I thought we agreed this doesn't make sense and we will not do this any more.}

\newcommand{\J}{\mathbf{1}}
\renewcommand{\t}{\tau}
\newcommand{\w}{\omega}
\newcommand{\p}{\pi}
\newcommand{\oo}{\infty}
\newcommand{\T}{{}^{\mathsf{T}}}
\renewcommand{\root}{g} %{\mathfrak{R}}
%\newpage
\section{Vertex-to-vertex paths on the cube} \label{sec:v2v}

In this section, we examine vertex-to-vertex paths on the cube
that each correspond to a segment from the origin $(0,0)$ to a lattice point $(a,b)$ in the first quadrant.
In \fig{orchard}, each ``sight line'' is drawn starting from a lattice point visible%
\footnote{\emph{visible}: Definition~\ref{def:visible}. We think of ``visibility'' as if standing at the corner of an orchard, with trees placed at grid points, as in \fig{orchard}.}
from the origin and extending to ``block'' other lattice points from view of the origin.
It is helpful to put a structure on the visible lattice points, called the Stern--Brocot tree.

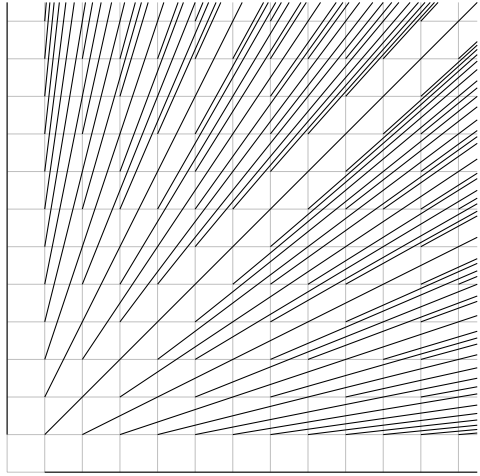
\begin{figure}[hbtp]
\begin{tikzpicture}[scale=0.5]
\draw [step=1cm, gray!50, very thin] (0,0) grid (12.5,12.5);
\draw (1,0) -- (12.5,0);
\draw (0,1) -- (0,12.5);

\draw ( 1,12) -- (12.5*1/12, 12.5); \draw ( 1,11) -- (12.5*1/11, 12.5); \draw ( 1,10) -- (12.5*1/10, 12.5); \draw ( 1, 9) -- (12.5*1/9, 12.5); \draw ( 1, 8) -- (12.5*1/8, 12.5); \draw ( 1, 7) -- (12.5*1/7, 12.5); \draw ( 1, 6) -- (12.5*1/6, 12.5); \draw ( 2,11) -- (12.5*2/11, 12.5); \draw ( 1, 5) -- (12.5*1/5, 12.5); \draw ( 2, 9) -- (12.5*2/9, 12.5); \draw ( 1, 4) -- (12.5*1/4, 12.5); \draw ( 3,11) -- (12.5*3/11, 12.5); \draw ( 2, 7) -- (12.5*2/7, 12.5); \draw ( 3,10) -- (12.5*3/10, 12.5); \draw ( 1, 3) -- (12.5*1/3, 12.5); \draw ( 4,11) -- (12.5*4/11, 12.5); \draw ( 3, 8) -- (12.5*3/8, 12.5); \draw ( 2, 5) -- (12.5*2/5, 12.5); \draw ( 5,12) -- (12.5*5/12, 12.5); \draw ( 3, 7) -- (12.5*3/7, 12.5); \draw ( 4, 9) -- (12.5*4/9, 12.5); \draw ( 5,11) -- (12.5*5/11, 12.5); \draw ( 1, 2) -- (12.5*1/2, 12.5); \draw ( 6,11) -- (12.5*6/11, 12.5); \draw ( 5, 9) -- (12.5*5/9, 12.5); \draw ( 4, 7) -- (12.5*4/7, 12.5); \draw ( 7,12) -- (12.5*7/12, 12.5); \draw ( 3, 5) -- (12.5*3/5, 12.5); \draw ( 5, 8) -- (12.5*5/8, 12.5); \draw ( 7,11) -- (12.5*7/11, 12.5); \draw ( 2, 3) -- (12.5*2/3, 12.5); \draw ( 7,10) -- (12.5*7/10, 12.5); \draw ( 5, 7) -- (12.5*5/7, 12.5); \draw ( 8,11) -- (12.5*8/11, 12.5); \draw ( 3, 4) -- (12.5*3/4, 12.5); \draw ( 7, 9) -- (12.5*7/9, 12.5); \draw ( 4, 5) -- (12.5*4/5, 12.5); \draw ( 9,11) -- (12.5*9/11, 12.5); \draw ( 5, 6) -- (12.5*5/6, 12.5); \draw ( 6, 7) -- (12.5*6/7, 12.5); \draw ( 7, 8) -- (12.5*7/8, 12.5); \draw ( 8, 9) -- (12.5*8/9, 12.5); \draw ( 9,10) -- (12.5*9/10, 12.5); \draw (10,11) -- (12.5*10/11, 12.5); \draw (11,12) -- (12.5*11/12, 12.5); \draw ( 1, 1) -- (12.5*1/1, 12.5); \draw (12,11) -- (12.5, 12.5*11/12); \draw (11,10) -- (12.5, 12.5*10/11); \draw (10, 9) -- (12.5, 12.5*9/10); \draw ( 9, 8) -- (12.5, 12.5*8/9); \draw ( 8, 7) -- (12.5, 12.5*7/8); \draw ( 7, 6) -- (12.5, 12.5*6/7); \draw ( 6, 5) -- (12.5, 12.5*5/6); \draw (11, 9) -- (12.5, 12.5*9/11); \draw ( 5, 4) -- (12.5, 12.5*4/5); \draw ( 9, 7) -- (12.5, 12.5*7/9); \draw ( 4, 3) -- (12.5, 12.5*3/4); \draw (11, 8) -- (12.5, 12.5*8/11); \draw ( 7, 5) -- (12.5, 12.5*5/7); \draw (10, 7) -- (12.5, 12.5*7/10); \draw ( 3, 2) -- (12.5, 12.5*2/3); \draw (11, 7) -- (12.5, 12.5*7/11); \draw ( 8, 5) -- (12.5, 12.5*5/8); \draw ( 5, 3) -- (12.5, 12.5*3/5); \draw (12, 7) -- (12.5, 12.5*7/12); \draw ( 7, 4) -- (12.5, 12.5*4/7); \draw ( 9, 5) -- (12.5, 12.5*5/9); \draw (11, 6) -- (12.5, 12.5*6/11); \draw ( 2, 1) -- (12.5, 12.5*1/2); \draw (11, 5) -- (12.5, 12.5*5/11); \draw ( 9, 4) -- (12.5, 12.5*4/9); \draw ( 7, 3) -- (12.5, 12.5*3/7); \draw (12, 5) -- (12.5, 12.5*5/12); \draw ( 5, 2) -- (12.5, 12.5*2/5); \draw ( 8, 3) -- (12.5, 12.5*3/8); \draw (11, 4) -- (12.5, 12.5*4/11); \draw ( 3, 1) -- (12.5, 12.5*1/3); \draw (10, 3) -- (12.5, 12.5*3/10); \draw ( 7, 2) -- (12.5, 12.5*2/7); \draw (11, 3) -- (12.5, 12.5*3/11); \draw ( 4, 1) -- (12.5, 12.5*1/4); \draw ( 9, 2) -- (12.5, 12.5*2/9); \draw ( 5, 1) -- (12.5, 12.5*1/5); \draw (11, 2) -- (12.5, 12.5*2/11); \draw ( 6, 1) -- (12.5, 12.5*1/6); \draw ( 7, 1) -- (12.5, 12.5*1/7); \draw ( 8, 1) -- (12.5, 12.5*1/8); \draw ( 9, 1) -- (12.5, 12.5*1/9); \draw (10, 1) -- (12.5, 12.5*1/10); \draw (11, 1) -- (12.5, 12.5*1/11); \draw (12, 1) -- (12.5, 12.5*1/12);
\end{tikzpicture}
\caption{Sight lines in an orchard.  \label{fig:orchard}}
\end{figure}

\begin{definition}[Stern--Brocot tree]
Consider the set
$V=\left\{ (a,b)\in\N\times\N\mid\gcd(a,b)=1\right\} $
of lattice points in the first quadrant that are visible from the origin.%
The \defn{Stern--Brocot tree} $T$ is a binary tree structure on~$V$,
recursively defined as follows.
For each $t\in V$, let $t^+$ and $t^-$ denote its \defn{positive parent} and \defn{negative parent}, respectively.
For the \defn{root node} $\root:=(1,1)$, let $\root^{+}:=(0,1)$ and $\root^{-}:=(1,0)$.

For $t\in V$ with $t^+$ defined,
let its \defn{positive child} $t_+$ be given by $t_+=t+t^+$.
The positive and negative parents of $t_+$ are given by $(t_+)^+=t^+$ and $(t_+)^-=t$, respectively.
Similarly, for $t\in V$ with $t^-$ defined,
its \defn{negative child} $t_-$ is given by $t_-=t+t^-$,
whose positive and negative parents are given by $(t_-)^+=t$ and $(t_-)^-=t^-$, respectively.
Note that $t=t^{-}+t^{+}$ for $t\in V$.

\begin{comment}
For $t\in V$ with $t^+$ defined,
consider the node $s:=t^++t\in V$.
Define the positive and negative parents of $s$ as $s^+:=t^+$ and $s^-:=t$, respectively.
This node $s$ is the \defn{positive child} of $t$, denoted $t_+$.
%
Similarly, for $t\in V$ with $t^-$ defined,
the node $s:=t^-+t\in V$ is the \defn{negative child} of $t$, denoted $t_-$.
The positive and negative parents of $s$ are $s^+:=t$ and $s^-:=t^-$, respectively.
Note that $t=t^{-}+t^{+}$ for $t\in V$.
\end{comment}

For each non-root $t\in V$, draw a \defn{positive edge} $(t,t_+)$ and a \defn{negative edge} $(t,t_-)$.
This establishes the binary tree structure~$T$.
Note that for a non-root $t\in V$, precisely one of $(t^+,t)$ and $(t^-,t)$ is an edge in~$T$.
By an abuse of notation, we write $t\in T$ instead of $t\in V$.
%(this could be visualized as a parallelogram having vertices in the lattice).

The \defn{depth} of an element $t\in T$ is the distance $d(t)$ in the tree from the root $\root$ to~$t$.
\end{definition}

Note that $\root^+=(0,1)$ and $\root^-=(1,0)$, the parents of the root node,
are not strictly considered to be part of the tree~$T$.
Let $\overline{T}:=T\cup\{ \root^+,\root^-\} $,
a ``completion'' which is useful for the recurrences discussed below.
Figure~\ref{sbt} shows the first five levels of the Stern--Brocot tree~$T$, depths $0$ through~$4$, along with the parents of the root;
when tracing from left to right, positive edges slant upwards and negative edges slant downwards.

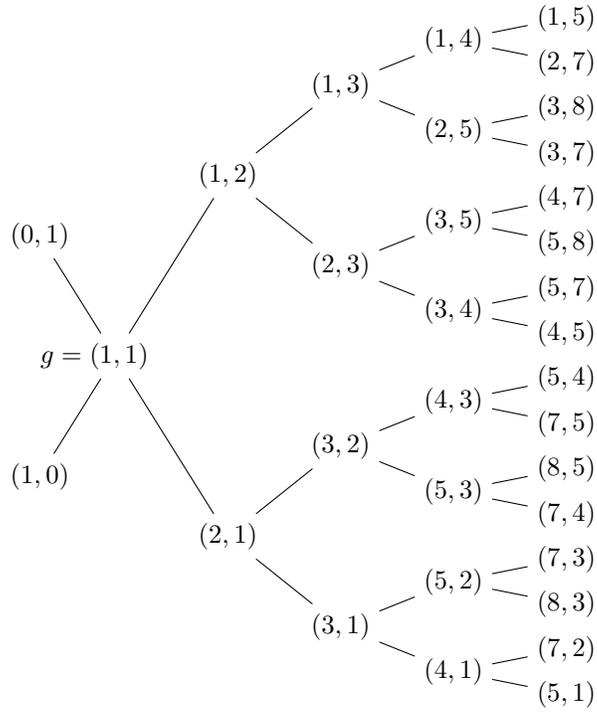
\begin{figure}[hbtp] 
\begin{tikzpicture}
  [grow=right,
   level distance=1.5cm,
   level 1/.style={sibling distance=-4.8cm},
   level 2/.style={sibling distance=-2.4cm},
   level 3/.style={sibling distance=-1.2cm},
   level 4/.style={sibling distance=-0.6cm},
   level 5/.style={sibling distance=-0.3cm}
  ]

  \node (positive-ancestor) {$(0,1)$};
  \node (root) [below=of positive-ancestor,xshift=1cm] {$\root=(1,1)\phantom{=\root}$}
    child {node {$(1,2)$}
      child {node {$(1,3)$}
        child {node {$(1,4)$}
          child {node {$(1,5)$}}
          child {node {$(2,7)$}}
        }
        child {node {$(2,5)$}
          child {node {$(3,8)$}}
          child {node {$(3,7)$}}
        }
      }
      child {node {$(2,3)$}
        child {node {$(3,5)$}
          child {node {$(4,7)$}}
          child {node {$(5,8)$}}
        }
        child {node {$(3,4)$}
          child {node {$(5,7)$}}
          child {node {$(4,5)$}}
        }
      }
    }
    child {node {$(2,1)$}
      child {node {$(3,2)$}
        child {node {$(4,3)$}
          child {node {$(5,4)$}}
          child {node {$(7,5)$}}
        }
        child {node {$(5,3)$}
          child {node {$(8,5)$}}
          child {node {$(7,4)$}}
        }
      }
      child {node {$(3,1)$}
        child {node {$(5,2)$}
          child {node {$(7,3)$}}
          child {node {$(8,3)$}}
        }
        child {node {$(4,1)$}
          child {node {$(7,2)$}}
          child {node {$(5,1)$}}
        }
      }
    }
  ;
  \node (negative-ancestor) [below=of root,xshift=-1cm] {$(1,0)$};
  \draw (positive-ancestor) -- (root) -- (negative-ancestor);
\end{tikzpicture}
\begin{quote}\caption{The Stern--Brocot tree showing depths $0,1,2,3$ and $4$.\label{sbt}} \end{quote}
\end{figure}

\begin{definition}
To each $t\in T$, we associate the following:
\begin{itemize}
\item $\t(t)$ = the tumble sequence associated to~$t$, a word on alphabet $\{r,u\}$.
\item $\s(t)$ = the associated cube orientation, an element of~$S_4$.
%\item $\w(t)$ = the associated traversal word in $T$, a word on alphabet $\{+,-\}$.
\end{itemize}
For $t=(a,b)\in T$, as $\gcd(a,b)=1$,
the line segment from $(0,0)$ to $(a,b)$ does not contain any other lattice points,
so there is a unique tumble sequence $\t(t)$ associated to~$t$ (see \fig{tumbletree}).

Tumbling the cube according to $\t(t)$ results in the cube orientation element $\s(t) \in S_4$.
%Finally, to get from $\root$ to $t$ in the tree we traverse paths of positive and negative children;
%we keep track of these with $+$ and $-$ via the \defn{traversal word}.
\end{definition}

\begin{comment}
\begin{figure}[hbtp] 
\begin{tikzpicture}
  [grow=right,
   level distance=1.5cm,
   level 1/.style={sibling distance=-4.8cm},
   level 2/.style={sibling distance=-2.4cm},
   level 3/.style={sibling distance=-1.2cm},
   level 4/.style={sibling distance=-0.6cm},
   level 5/.style={sibling distance=-0.3cm}
  ]

  \node (root) [below=of positive-ancestor] {$1_{T}$}
    child {node {$+$}
      child {node {$++$}
        child {node {$+++$}
          child {node {$++++$}}
          child {node {$+++ -$}}
        }
        child {node {$++ -$}
          child {node {$++ - +$}}
          child {node {$++ --$}}
        }
      }
      child {node {$+ -$}
        child {node {$+ - +$}
          child {node {$+ - ++$}}
          child {node {$+ - + -$}}
        }
        child {node {$+ --$}
          child {node {$+ -- +$}}
          child {node {$+ ---$}}
        }
      }
    }
    child {node {$-$}
      child {node {$- +$}
        child {node {$- ++$}
          child {node {$- +++$}}
          child {node {$- ++ -$}}
        }
        child {node {$- + -$}
          child {node {$- + - +$}}
          child {node {$- + --$}}
        }
      }
      child {node {$--$}
        child {node {$-- +$}
          child {node {$-- ++$}}
          child {node {$-- + -$}}
        }
        child {node {$---$}
          child {node {$--- +$}}
          child {node {$----$}}
        }
      }
    }
  ;
\end{tikzpicture}
\begin{quote}\caption{The Stern--Brocot tree showing traversal words. \label{traversaltree}} \end{quote}
\end{figure}
\end{comment}

\begin{figure}[hbtp] 
\begin{tikzpicture}
  [grow=right,
   level distance=1.8cm,
   level 1/.style={sibling distance=-4.8cm},
   level 2/.style={sibling distance=-2.4cm},
   level 3/.style={sibling distance=-1.2cm},
   level 4/.style={sibling distance=-0.6cm},
   level 5/.style={sibling distance=-0.3cm}
  ]

  \node (root) [below=of positive-ancestor] {$1_{T}$}
    child {node {$u$}
      child {node {$uu$}
        child {node {$uuu$}
          child {node {$uuuu$}}
          child {node {$uuuruuu$}}
        }
        child {node {$uuruu$}
          child {node {$uuruuuruu$}}
          child {node {$uuruuruu$}}
        }
      }
      child {node {$uru$}
        child {node {$uruuru$}
          child {node {$uruuruuru$}}
          child {node {$uruururuuru$}}
        }
        child {node {$ururu$}
          child {node {$ururuururu$}}
          child {node {$urururu$}}
        }
      }
    }
    child {node {$r$}
      child {node {$rur$}
        child {node {$rurur$}
          child {node {$rururur$}}
          child {node {$rururrurur$}}
        }
        child {node {$rurrur$}
          child {node {$rurrururrur$}}
          child {node {$rurrurrur$}}
        }
      }
      child {node {$rr$}
        child {node {$rrurr$}
          child {node {$rrurrurr$}}
          child {node {$rrurrrurr$}}
        }
        child {node {$rrr$}
          child {node {$rrrurrr$}}
          child {node {$rrrr$}}
        }
      }
    }
  ;
\end{tikzpicture}
\begin{quote}\caption{The Stern--Brocot tree showing tumble sequences. \label{fig:tumbletree}} \end{quote}
\end{figure}
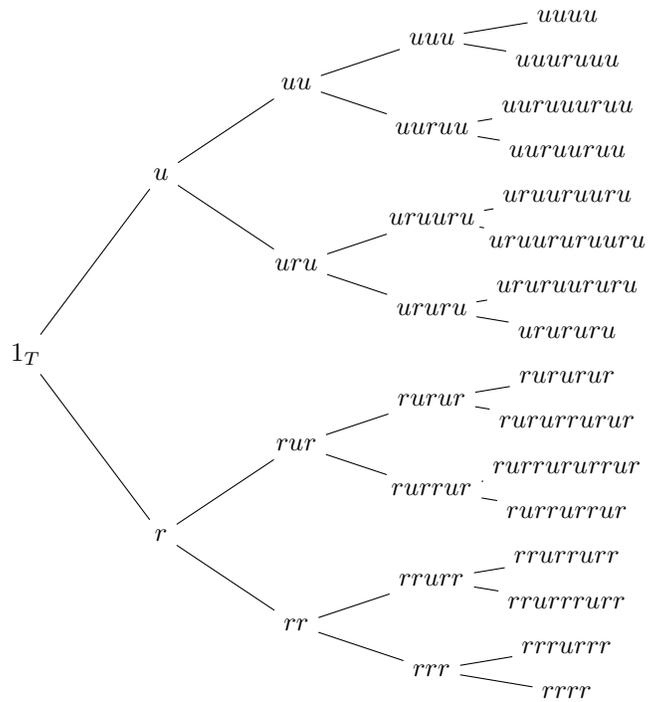

\begin{definition}
If $(t^+,t)$ are adjacent, the \defn{positive branch} of $t$ is the set of vertices $b_+(t):=\{s\in T:s^+=t^+\}$ that are positive descendants of~$t$.
Analogously, if $(t^-,t)$ are adjacent, the \defn{negative branch} is given by $b_-(t):=\{s\in T:s^-=t^-\}$.
For any non-root $t\in T$, precisely one of $b_+(t)$ and $b_-(t)$ is defined, which we denote $b(t)$ for convenience.
\end{definition}
For example, the positive branch $b_+(r)$ of $r$ is $\{rur, rurur, rururur, \dotsc\}$ 
and the negative branch $b_-(r)$ is not defined as $r$ is a negative child (see \fig{tumbletree}).

\renewcommand{\a}[1]{\gen{#1}}
We include the following standard result regarding the Stern--Brocot tree for completeness.
For $t,t'\in T$, let $P(t,t')$ denote the parallelogram with vertices $(0,0)$, $t$, $t'$, and $t+t'$.
%, and let $a(t,t')$ denote its area.
\begin{lemma} \label{lem:empty}
Let $t\in T$.
The parallelograms $P(t,t^+)$ and $P(t,t^-)$ contain no lattice points in their interiors.
\end{lemma}
\begin{proof}
By Pick's theorem, if the area of the lattice polygon $P(t,t')$ is $1$,
then $P(t,t')$ contains no interior lattice points (and no lattice points on the boundary besides its four vertices).
For $t=(x,y)$ and $t'=(x',y')$, let $\a{t,t'}=xy'-yx'$.
Note that $\a{\cdot,\cdot}$ is a skew-symmetric bilinear form,
and $\abs{\a{t,t'}}$ is the area of the parallelogram $P(t,t')$.
It suffices to prove that $\abs{\a{t,t^+}}=\abs{\a{t,t^-}}=1$.

We prove this by induction on the depth~$d(t)$.
The base case where $t$ is the root $\root=(1,1)$ is trivial.

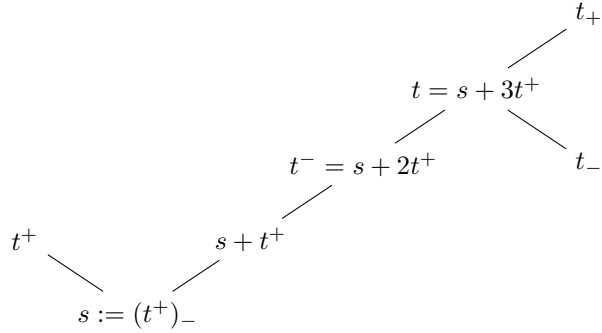
\begin{figure}[hbtp]
\begin{tikzpicture}[xscale=1.5,yscale=-1]
\path (0,0)
   node (A) {$t^+$}
   ++(1,+1) node (B) {$s:=(t^+)_-$}
   ++(1,-1) node (C) {$s+t^+$}
   ++(1,-1) node (D) {$t^-=s+2t^+$}
   ++(1,-1) node (E) {$t=s+3t^+$}
   +(1,+1) node (F) {$t_-$}
   +(1,-1) node (G) {$t_+$}
   ;
\draw (A) -- (B) -- (C) -- (D) -- (E) -- (F)
   (E) -- (G);
\end{tikzpicture}
\caption{The subtree referenced in Lemma~\ref{lem:empty}, with $k=3$. \label{fig:para-prop}}
\end{figure}
Suppose $t$ is not the root~$\root$.
Without loss of generality, suppose $t$ is a positive child, so $t$ is adjacent to~$t^-$ but not to~$t^+$ (see \fig{para-prop}).
Let $s:=(t^+)_-$ be the negative child of the positive parent of~$t$.
By definition, $t$ is the $k$th positive child of $s$ for some $k\geq1$,
and so $t=s+kt^+$.
By bilinearity, we get
$$\a{t,t^+}=\a{s+kt^+,t^+}=\a{s,t^+}+k\a{t^+,t^+}=\a{s,t^+}.$$
As $d(s)<d(t)$, by induction we have $\abs{\a{s,s^+}}=1$.
Because $s^+=t^+$, we obtain $\abs{\a{t,t^+}}=1$, as desired.
Using skew-symmetry, we obtain
\begin{align*}
\a{t,t^-}&=\a{s+kt^+,s+(k-1)t^+}=\a{s,s}+k\a{t^+,s}+(k-1)\a{s,t^+}+k(k-1)\a{t^+,t^+}\\
&=-k\a{s,t^+}+(k-1)\a{s,t^+}=-\a{s,t^+},
\end{align*}
so $\abs{\a{t,t^-}}=\abs{\a{s,t^+}}=1$, as desired.
\end{proof}

We use the result above to derive a recurrence of tumble sequences.
Note that $\t(t^+)$ and $\t(t^-)$ are both defined if and only if $t\notin b_+(\root)\cup b_-(\root)$.

\begin{lemma}[Tumble recurrence] \label{lem:recurrence}
Let $t\in T$ and $\root$ the root.
If $t\notin b_+(\root)\cup b_-(\root)$, then
$$ \t(t)=\t(t^+)ru\t(t^-)=\t(t^-)ur\t(t^+).$$
\end{lemma}
\begin{proof}
We prove the first equality.
The second is similar and follows from symmetry.

Without loss of generality, suppose $t$ is a positive child, so $t$ is adjacent to~$t^-$ but not to $t^+$.
As $(t^-)^+=t^+$,
by Lemma~\ref{lem:empty},
the parallelogram $P:=P(t^-,t^+)$ contains no interior lattice points.

\renewcommand{\L}[2]{L_{#1}^{#2}}
Let $t^+=(a,b)$.
Let $a',b'\in\R$ such that $(a,b')$ and $(a',b)$ lie on the geodesic line segment $L$ from $(0,0)$ to~$t$;
note that $a<a'<a+1$ and $b-1<b'<b$
(see Figure~\ref{fig:homotopy}).

For $v,v'\in\R^2$, let $\L{v}{v'}$ denote the (directed) line segment from $v$ to~$v'$ and,
by an abuse of notation,
write $\t(\L{v}{v'})$ for the sequence of the letters $r$ and $u$ corresponding to crossing a vertical and a horizontal edge, respectively,
while traversing along this directed line segment.

By breaking $\L{(0,0)}t$ at points $(a,b')$ and $(a',b)$, we get
%Note that traversing from the origin to $t$ involves traversing along $\L{(0,0)}t$ to $(a,b')$,
%crossing a vertical edge at $x=a$, crossing a horizontal edge at $y=b$,
%and then traversing straight to~$t$.
%As such, we get
$$\t(t)=\t(\L{(0,0)}t)=\t(\L{(0,0)}{(a,b')})\cdot r\cdot\t(\L{(a,b')}{(a',b)})\cdot u\cdot\t(\L{(a',b)}t),$$
where the $r$ and the $u$ correspond to
crossing a vertical edge at $(a,b')$ and a horizontal edge at $(a',b)$, respectively.
Now $\L{(a,b')}{(a',b)}$ crosses no edges as it is a segment inside a lattice square.
It remains to calculate the other two tumble sequences.

Consider the homotopy $F(z)=\L{(0,0)}{(a,zb+(1-z)b')}$
of line segments from $F(0)=\L{(0,0)}{(a,b')}$ to $F(1)=\L{(0,0)}{(a,b)}$.
%Consider the line segments $\L{(0,0)}{(a,b)}$ and $\L{(0,0)}{(a,b')}$.
Since there are no lattice points in the parallelogram~$P$,
and we are moving an endpoint along an edge,
%which contains~$\L{(0,0)}{(a,b')}$,
$\t\circ F(z)$ is constant for $z\in[0,1]$,
and so $$\t(\L{(0,0)}{(a,b')})=\t(\L{(0,0)}{(a,b)})=\t(t^+).$$
Similarly, the homotopy $G(z)=\L{(za+(1-z)a',b)}{t}$
shows the first equality in 
$$\t(\L{(a',b)}t)=\t(\L{(a,b)}t)=\t(\L{(0,0)}{t-t^+})=\t(\L{(0,0)}{t^-})=\t(t^-),$$
where the second equality follows from invariance of $\t$ under lattice translation.
This completes the proof.
\end{proof}

\begin{figure}[hbtp]
\begin{tikzpicture}[scale=2]
\coordinate (add) at (0.1,0.1);
\coordinate (O) at (0,0);
\coordinate [label=right:{~$t=(3,2)$}] (t) at (3,2);
\coordinate [label=above left:{$t^+=(a,b)$}] (t+) at (1,1);
\draw ($(O)-(add)$) grid ($(t)+(add)$);
\draw [name path=O--t] (O) -- (t);
\path [name path=t+x] (t+) -- +(0,-1);
\path [name path=t+y] (t+) -- +(1,0);
\path [name intersections={of=O--t and t+x,by={[label=below right:{$(a,b')$}]b'}}];
\path [name intersections={of=O--t and t+y,by={[label=below right:{$(a',b)$}]a'}}];
\foreach \num in {0,25,5,75}
{
   \coordinate (z\num) at ($ (t+)!.\num!(b') $);
   \draw (O) -- (z\num) -- (t);
};
\foreach \point in {t,t+,a',b'}
   \fill [black,opacity=.7] (\point) circle (1pt);
\end{tikzpicture}
\caption{The homotopy construction in Lemma~\ref{lem:recurrence} for $t=(3,2)$, with $t^+=(a,b)=(1,1)$. \label{fig:homotopy}}
\end{figure}
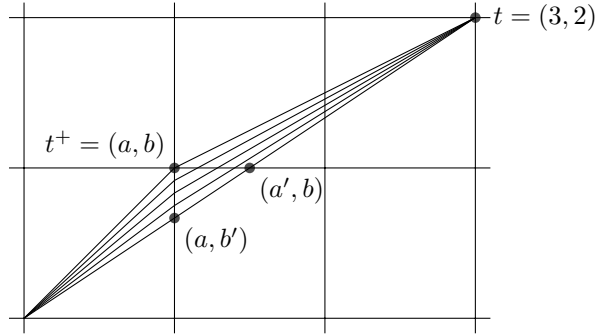

%We formally extend $\s$ to the parents of the root $\root\in T$ 

\begin{definition}
Extend $\s$ to the parents of the root $\root\in T$ by formally defining $\s(\root^+)=r\i$ and $\s(\root^-)=u\i$,
where $r=1230$ and $u=2310$ are considered as elements of~$S_4$ (see Definition~\ref{basicdefs}).
\end{definition}

\begin{rmk}
Since $\t(t)\mapsto\s(t)$ is a monoid homomorphism,
$\s(t)$ enjoys the above recurrence, restated as a corollary below.
In fact, by our judicious definition of $\s(\root^+)$ and $\s(\root^-)$,
the recurrence works for all $t\in T$, \latin{i.e.}, for $t\in b_+(\root)\cup b_-(\root)$ as well.
We choose not to do this for $\t$ for the following reason.
As $\t$ maps to words, \latin{i.e.}, elements of the free monoid, on the formal symbols $\{r,u\}$,
technically, $r\i$ and $u\i$ do not make sense in this setting.
\end{rmk}

\begin{cor}[Orientation recurrence] \label{cor:recurrence}
For $t\in T$, we have
$$ \s(t)=\s(t^+)ru\s(t^-)=\s(t^-)ur\s(t^+).$$
\end{cor}
\begin{proof}
%We prove the first equality; the second is similar.
For $t\notin b_+(\root)\cup b_-(\root)$, the recurrence formula follows immediately from that of Lemma~\ref{lem:recurrence}.
Otherwise, without loss of generality, suppose $t\in b_+(\root)$.
For $t=(1,k+1)$ the $k$th positive child of~$\root$,
the geodesic from $(0,0)$ to $(1,k+1)$ crosses $k$ horizontal edges, so $\s(t)=u^k$.
As the recurrence for $t=\root$ is immediate,
we assume $k>0$.
We have $\s(t^-)=u^{k-1}$ as above and $\s(t^+)=\s(\root^+)=r\i$ by definition,
so $$\s(t^+)ru\s(t^-)=r\i r u u^{k-1}=u^k=\s(t),$$
as desired.
The second equality is analogous.
\end{proof}

\begin{proposition}
The sequence of cube orientations associated to a branch is periodic with period at most~$4$.
\end{proposition}
\begin{proof}
The sequences of cube orientations of the two branches $b_+(\root)$ and $b_-(\root)$ starting at the root are $\{u^k\}_k$ and $\{r^k\}_k$, respectively,
and so are $4$-periodic.
Otherwise, let $t\in T$ be at depth at least~$1$ and consider the branch $b(t)$.
Without loss of generality, assume $b(t)$ is a positive branch.
Let $t_k$ be the $k$th positive child of~$t=t_0$, which are precisely the elements of $b(t)$ in sequence.
We have $$\s(t_k)=\s(t_{k-1})ru\s(t^-)=\s(t)[ru\s(t^-)]^k$$
by Corollary~\ref{cor:recurrence} and induction.
Since $ru\s(t^-)$ is an element of~$S_4$, its order is at most~$4$,
and hence the sequence $\{\s(t_k)\}_k$ is periodic with period at most~$4$.
\end{proof}

\begin{definition}
%parent-child pairs / edge type $(\s(t),\s(t_+),+)$ and $(\s(t),\s(t_-),-)$
%note to self: $(\s(t^+),\s(t),-)$ doesn't work since $t^+$ and $t$ are not necessarily adjacent.
Each edge of the tree contains a parent $t$ on the left endpoint,
and either the positive child $t_+$ or the negative child $t_-$ on the right endpoint.
We label each edge, and thus each parent-child pair, by whether the child is positive or negative:
$(\s(t),\s(t_+),+)$ or $(\s(t),\s(t_-),-)$, respectively.
Each edge is thus associated to an element of $S_4 \times S_4 \times \{+,-\}$, called its \defn{label}.
\end{definition}

\begin{lem}\label{lem:children}
Each edge label uniquely determines the labels of its two child edges.
\end{lem}
\begin{proof}
Without loss of generality, consider a positive edge $(\s(s),\s(t),+)$.
Since $s$, $t$, and $t_+$ are consecutive elements on the same branch, the cube orientations are $x,xy,xy^2$ for some $x,y\in S_4$,
and hence we have $$\s(t_+)=\s(t)\s(s)\i\s(t).$$
On the other hand, as $(t_-)^+=t$ and $(t_-)^-=s$, we have $$\s(t_-)=\s(t)ru\s(s)$$
by Corollary~\ref{cor:recurrence}, as desired.
\end{proof}

\begin{definition}
Let $S=\{\s(t)\in S_4\mid t\in T\}$ be the subset of cube orientations that occur in~$T$,
and let $E\subseteq S_4\times S_4\times\{+,-\}$ be the edge labels that occur.
\end{definition}

\begin{lem} \label{lem:saturate}
We have $\abs{S}=9$ and $\abs{E}=54$.
In particular, we have $$S=\{0123, 0213, 1032, 1230, 2301, 2310, 3012, 3120, 3201\}.$$
\end{lem}
\begin{proof}
Starting from the two root edges, we may recursively compute all child edges to some depth.
The proof of Lemma~\ref{lem:children} states clearly how to compute the edge labels of its two child edges.
Let $E'$ be the labels of edges that occur by depth~$7$.
By direct computation, we have $\abs{E}\geq\abs{E'}=54$.
Appendix~\ref{sec:saturate} lists the labels of the two child edges for each edge label in~$E'$.
One can check that all the child edge labels are in~$E'$.
This proves that $E=E'$, as desired.
Reading the cube orientations occurring in $E$ yields the $9$ orientations in~$S$.
\end{proof}
\begin{rmk}
Alternatively,
the above computation can be confirmed simply by calculating the tree to depth $8$ and observing that no new edge labels occur past depth~$7$.
\end{rmk}

Because the cube orientations $1320$, $2130$, and $3210$ are not in~$S$,
there are no geodesics ending at vertex~$0$.
\begin{corollary} \label{cor:cube-noloop}
The cube does not admit geodesic loops starting and ending at the same vertex.
\end{corollary}

Let $A$ be a $54\times54$ matrix, with rows and columns indexed by~$E$.
For $e,f\in E$, the row $e$ column $f$ entry of $A$ is given by
$$A_{ef}=\begin{cases}1/2 & \text{if $f$ is a child of $e$,} \\ 0 & \text{otherwise.}\end{cases}$$

Let $\J=(1,1,\ldots,1)$.
Note that $A\J\T=\J\T$ since each row sums to~$1$.
By direct computation (see Appendix~\ref{sec:saturate}), each column sums to~$1$, and so $$\J A=\J.$$
In other words, $\p=\tfrac{1}{54}\J$ is a stationary distribution.

\begin{lem}
Let $M=\displaystyle\lim_{k\to\oo}A^k$.
The limit exists and $M=\J\p$.
\end{lem}
\begin{proof}
By direct computation, $A$ has (left) eigenvalue $1$ with multiplicity~$1$,
so the underlying directed graph is strongly connected.
All other eigenvalues have absolute values strictly less than~$1$,
so the system is aperiodic.
The result follows.
\end{proof}

In fact, regardless of the starting state, the distribution converges to the unique stationary distribution~$\p$,
because $A$ is a regular stochastic matrix and $\p$ is a steady-state vector for $A$,
so $\p$ is the unique attracting long-term state of the system.

\begin{theorem} \label{thm:proportions}
In the limit, the proportions for the cube orientations and associated vertices are as shown in the following table.
Vertex $7$ thus has a frequency of $12/54$.

\begin{table}[h]
\begin{tabular}{|c|c|c|cl}
\cline{1-3}
 $\s$ &  vertex & frequency &  &  \\ \cline{1-3}
3201 & 1 & 8/54 &  &  \\ \cline{1-3}
3012 & 2 & 8/54 &  &  \\ \cline{1-3}
0123 & 3 & 6/54 &  &  \\ \cline{1-3}
0213 & 4 & 8/54 &  &  \\ \cline{1-3}
1032 & 5 & 6/54 &  &  \\ \cline{1-3}
2301 & 6 & 6/54 &  &  \\ \cline{1-3}
1230 & 7 & 4/54 &  &  \\ \cline{1-3}
2310 & 7 & 4/54 &  &  \\ \cline{1-3}
3120 & 7 & 4/54 &  &  \\ \cline{1-3}
\end{tabular}
\end{table}
\end{theorem}

%\begin{proof}
%Cube orientations are associated to ending vertices as follows:
%$$3201\leftrightarrow 1; \hspace{1em} 
%3012 \leftrightarrow 2;  \hspace{1em} 
%0123 \leftrightarrow 3;  \hspace{1em} 
%0213 \leftrightarrow 4;  \hspace{1em} 
%1032 \leftrightarrow 5;  \hspace{1em} 
%2301 \leftrightarrow 6;  \hspace{1em} 
%1230,2310,3120 \leftrightarrow 7.$$
%\end{proof}

\renewcommand{\b}{\mathbf{b}}
\newcommand{\sv}{\mathbf{s}_v}
Let $\b$ be the $1\times 54$ row vector indexed by $E$,
with a $0$ in each entry except for a $1/2$ in the two entries corresponding to the edges between the root $\root$ and its two children.
For the $7$ vertices $v\in\{1,\ldots,7\}$, let $\sv$ be the $54\times 1$ column vector indexed by $E$,
with a $0$ in each entry except for a $1$ in the entry corresponding to edge labels whose child cube orientation corresponds to vertex~$v$.
%$6$, $8$ or, $12$ entries (this number depending on the frequencies in Theorem~\ref{thm:proportions}) where the associated edge in $A$ has vertex $v$ as the child.
\begin{cor} \label{ways}
At depth $k\geq1$ in the tree~$T$,
the frequency of vertex $v$ is given by the number $\b A^{k-1}\sv$.
The number of ways of getting there is the number $(2\b)(2A)^{k-1}\sv$.
\end{cor}

%\begin{question}
%(Dream) Say what kind of branches have certain orientations---which (parent,kid) makes those branches.
%\end{question}

\begin{remark}
Theorem \ref{thm:proportions} shows the frequencies of vertices at a given depth in the Stern--Brocot tree.
This is reasonably well-correlated with the lengths of trajectories from the origin:\
experimentally, we computed the frequencies of end vertex labels on lattice points over large square patches of the first quadrant with a corner at the origin and obtained the same proportions.
\end{remark}

%\newpage
\section{Face-to-vertex paths on the cube} \label{sec:f2v}
In the previous section, we considered vertex-to-vertex geodesics.
The more general question that we seek to answer is:
Given a point $p$ on a face of the cube and a vertex~$v$,
describe all possible geodesics from $p$ to~$v$.
We do not consider paths that pass through vertices before reaching~$v$.

\subsection{Visibility of vertices in the square lattice}
Given an initial point $p=(x,y)\in[0,1)^2$ in the starting square
and an ending vertex $(n,m)\in\Z_+^2$,
the problem is arguably split into two parts: visibility and vertex determination.
We need to know if $(n,m)$ is in the line of sight of $(x,y)$ before hitting another vertex and,
if so, what vertex is on $(n,m)$ following the tumble path containing the segment starting at $(x,y)$ and ending at $(n,m)$.

For visibility, there is a simple characterization.
Recall Definition~\ref{denominator} of the \defn{denominator} $\denom x$.
Setting $\mathbf{u}=(1,0)$ and $\mathbf{v}=(0,1)$, Lemma~\ref{lem:denom} becomes the following.
\begin{lemma} \label{latticevisibility}
The point $(n,m)$ is visible from $(x,y)$ if and only if $\denom{\frac{m-y}{n-x}}\leq\frac1n$.
\end{lemma}
\begin{comment}
In any case, this is simply saying that the fraction is not a rational $\frac pq$ such that $q<n$,
which would be precisely the case when $(n,m)$ is blocked by $(n-q,m-p)$.
\end{comment}
If $(x,y)=(0,0)$, the condition is that the fraction $\frac mn$ is reduced.
If $x\in\Q$ and $y\notin\Q$, then the fraction $\frac{m-y}{n-x}$ is irrational hence the denominator is~$0$.
As such, we have the following special cases:
\begin{enumerate}
\item If $(x,y)=(0,0)$ is at the origin, $(n,m)$ is visible if and only if $\gcd(n,m)=1$.
\item If $x\in\Q$ and $y\notin\Q$, or vice versa, $(n,m)$ is visible.%
\footnote{Even if $x$ and $y$ are incommensurable, \latin{i.e.}, $x/y\notin\Q$, $(n,m)$ could be blocked.  For example, for $(x,y)=(\sqrt3/3,\sqrt3/9+1/3)$, the lattice point $(n,m)=(5,2)$ is blocked by $(2,1)$, but $y/x=\sqrt3/3+1/3\notin\Q$.}
\end{enumerate}
From a generic starting point, all lattice points are visible.

\subsection{Constructing the tumble sequence for a face-to-vertex geodesic}

\begin{definitions}
%The \defn{first octant} is the set $\{(x,y):0\leq x \leq y\}$, the sector between $\theta=0$ and $\theta=\pi/4$.
For a lattice point $q=(n,m)$ with $0\leq n\leq m$,
%For a lattice point $(x,y)$ in the first octant, 
let $c(n,m)$ denote the convex quadrilateral whose vertices are $(1,0)$, $(0,0)$, $(0,1)$, and $(n,m)$,
and let $P(n,m)=\big(c(n,m)\cap\Z^2\big)\bs\{(0,0),(n,m)\}$
be the set of lattice points in $c(n,m)$ with two points removed.

Associate to each lattice point $p=(x,y)\in P(n,m)$ 
the line $\segment$ through $p$ and~$q$.
Let $\l_0,\l_1,\dotsc,\l_{t+1}$ be these lines such that 
their corresponding slopes
%$s_0,s_1,\dotsc,s_{t+1}$ 
are in increasing order.
Note that a line $\l_i$ may contain multiple lattice points.
For $i=0,\ldots,t$, let $z_i=z_i(n,m)$ be the region of $[0,1]^2$ strictly between $\ell_i$ and~$\ell_{i+1}$.

%are the open triangular regions within the cone $c(x,y)$ that contain no lattice points. Each zone is bounded by two lines, where each line passes through $(x,y)$ an interior point in the cone. The union of the closures of the zones is the cone.

We define tumble sequences $\tau_i=\tau_i(n,m)$ inductively for $i=0,1,\dotsc,t$.
Let $\tau_0=\tau_0(n,m)$ be the tumble sequence consisting of $n+m-2$ letters
with a $u$ in each position $u_j=n+m-j-\lceil \frac{n j}{m-1} \rceil$ for $j=1,2,\dotsc,m-1$
and an $r$ in each remaining position,
as in Lemma~\ref{tumbleseq}.

Having defined $\tau_{i-1}$, we obtain $\tau_i$ from it as follows.
For each lattice point $(x,y)\in P(n,m)$ on the line $\ell_i$,
replace the $ur$ in positions $x+y-1$ and $x+y$ with $ru$.
Call the resulting sequence~$\tau_i$.
\end{definitions}

\begin{proposition} \label{p2vsequence}
For any point in the region $z_i(n,m)$, the tumble sequence to the point $(n,m)$ is $\tau_i$.
\end{proposition}

\begin{proof}
The above construction sweeps out all the geodesics through $(n,m)$ that lie between the geodesic through $(0,1)$ and the geodesic through $(1,0)$. The tumble sequence changes exactly when the geodesic passes through a vertex, and the change in the tumble sequence is to change $ur$ to $ru$ in the place in the tumble word corresponding to where the geodesic passed across the vertex.
\end{proof}

\begin{example} \label{53ex}
We calculate the tumble sequences for the point $(5,3)$ (see Figure~\ref{anypoint}). 
By definition, we have $P(5,3)=\{(0,1),(1,1),(3,2),(2,1),(1,0)\}$,
where $(0,1)$ is on line $\ell_0$, $(1,1)$ and $(3,2)$ are on line $\ell_1$,
$(2,1)$ is on line $\ell_2$, and $(1,0)$ is on line $\ell_3$.
The lines separate $[0,1]^2$ into three regions $z_0$, $z_1$, and~$z_2$.

The tumble sequence $\tau_0(5,3)$ has $5+3-2=6$ letters.
The location of the $u$s are given by
$u_1=5+3-1-\lceil \frac{5\cdot1}{3-1} \rceil=4$ and $u_2=5+3-2-\lceil \frac{5\cdot2}{3-1} \rceil=1$,
so we have $$\tau_0(5,3)=urrurr.$$
Next, note that on line $\ell_1$, there are two lattice points: $(1,1)$ and $(3,2)$.
To obtain $\tau_1$ from $\tau_0$,
swap the underlined letters $ur$ ending at positions $1+1=2$ and $3+2=5$:
$$\tau_0(5,3)=\underline{ur}r\underline{ur}r \implies \tau_1(5,3)=\underline{ru}r\underline{ru}r.$$
Finally, note that on line $\ell_2$, there is one lattice point: $(2,1)$.
To obtain $\tau_2$ from $\tau_1$,
swap the underlined letters $ur$ ending at position $2+1=3$:
$$\tau_1(5,3)=r\underline{ur}rur\implies \tau_2(5,3)=r\underline{ru}rur.$$
The tumble sequence for every point in $z_0(5,3)$ is $urrurr$; the tumble sequence for every point in $z_1(5,3)$ is $rurrur$, and the tumble sequence for every point in $z_2(5,3)$ is $rrurur$.

It is easy to check in Figure~\ref{anypoint} that these are correct.
\end{example}

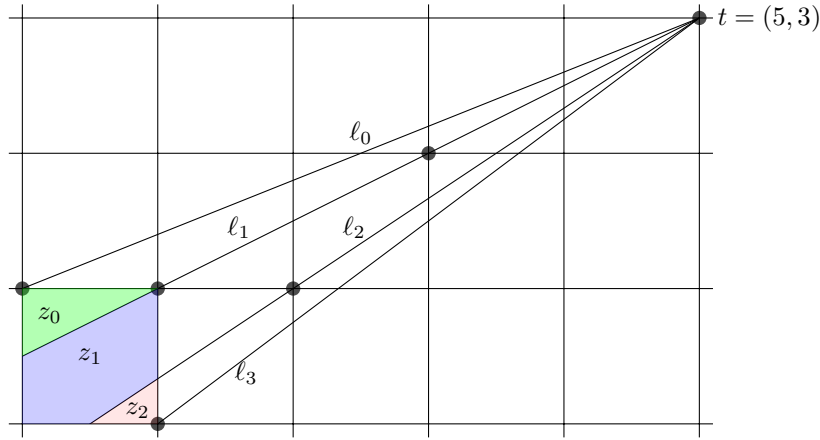
\begin{figure}[hbtp] 
\begin{tikzpicture}[scale=1.8]
\coordinate (add) at (0.1,0.1);
\coordinate (O) at (0,0);
\coordinate [label=right:{~$t=(5,3)$}] (t) at (5,3);
\coordinate (A) at (0,1);
\coordinate (B) at (1,1);
\coordinate (C) at (3,2);
\coordinate (D) at (2,1);
\coordinate (E) at (1,0);

\draw ($(O)-(add)$) grid ($(t)+(add)$);
\foreach \point in {t,A,B,C,D,E}
   \fill [black,opacity=.7] (\point) circle (1.5pt);

\clip (0,0) rectangle (5,3);
\draw (t) -- node[above] {$\ell_0$} (A);
\draw (t) -- node[above] {$\ell_1$} ($ (t) ! 1.7 ! (B) $); % also through (C)
\draw (t) -- node[above] {$\ell_2$} ($ (t) ! 1.7 ! (D) $);
\draw (t) -- node[right, very near end] {$\ell_3$} ($ (t) ! 1 ! (E) $);

\clip (0,0) rectangle (1,1);
\fill [green,opacity=.3] (A) -- (t) -- ($ (t) ! 1.7 ! (B) $);
\fill [blue,opacity=.2] ($ (t) ! 1.7 ! (D) $) -- (t) -- ($ (t) ! 1.7 ! (B) $);
\fill [red,opacity=.1] ($ (t) ! 1.7 ! (D) $) -- (t) -- (E);
\node at (0.2,0.8) {$z_0$};
\node at (0.5,0.5) {$z_1$};
\node at (0.85,0.1) {$z_2$};
\end{tikzpicture}
\begin{quote}\caption{Computing the tumble paths to $(5,3)$. \label{anypoint}} \end{quote}
\end{figure}

\section{Further questions}

For the regular tetrahedron, Theorem~\ref{thm:ptet} gives a description of all the directions, from an arbitrary starting point, that end at a given vertex of the tetrahedron.
We did not do this for the cube, which is a natural next step.
The main ingredient missing is a result analogous to Lemma~\ref{lem:label}, assigning vertex labels to lattice points for arbitrary starting points.
Solving this problem for the other regular polyhedra (octahedron, icosahedron, and dodecahedron) are also natural extensions of this work. 

Another natural direction to extend our work on the cube is to consider geodesics on general rectangular boxes, which do not tile the plane by unfolding.

\subsection*{Acknowledgements} Our interest in this problem began at the American Mathematical Society's Mathematics Research Communities workshop on Discrete and Computational Geometry in June 2012 at Snowbird. We thank Satyan Devadoss, Vida Dujmovic, Joseph O'Rourke, and Yusu Wang for organizing this workshop, and especially Joseph O'Rourke for introducing us to this problem. Travel funding for additional collaborations, during which we continued our work on this problem, was provided by the American Mathematical Society.

{\footnotesize \bibliography{thesis}}

\appendix
\section{Details for Lemma~\ref{lem:saturate}} \label{sec:saturate}
Each row of the table below contains the child edge labels of an edge label in~$E'$.

\begin{displaymath}
\small
\begin{array}{ccc}
(\s(s),\s(t),s)\in E' & (\s(t),\s(t_-),-) & (\s(t),\s(t_+),+)\\
\hline
(0123,0213,-) & (0213,0123,-) & (0213,3201,+) \\
(0123,0213,+) & (0213,3012,-) & (0213,0123,+) \\
(0123,1230,-) & (1230,2301,-) & (1230,0213,+) \\
(0123,2310,+) & (2310,0213,-) & (2310,1032,+) \\
(0123,3012,+) & (3012,2310,-) & (3012,2301,+) \\
(0123,3201,-) & (3201,1032,-) & (3201,1230,+) \\
(0213,0123,-) & (0123,0213,-) & (0123,3012,+) \\
(0213,0123,+) & (0123,3201,-) & (0123,0213,+) \\
(0213,1032,+) & (1032,2310,-) & (1032,3120,+) \\
(0213,1230,-) & (1230,3201,-) & (1230,0123,+) \\
(0213,2301,-) & (2301,3120,-) & (2301,1230,+) \\
(0213,2310,+) & (2310,0123,-) & (2310,3012,+) \\
(0213,3012,-) & (3012,2310,-) & (3012,2301,+) \\
(0213,3201,+) & (3201,1032,-) & (3201,1230,+) \\
(1032,0213,-) & (0213,2301,-) & (0213,2310,+) \\
(1032,2310,-) & (2310,0123,-) & (2310,3012,+) \\
(1032,3012,-) & (3012,1032,-) & (3012,0213,+) \\
(1032,3012,+) & (3012,3201,-) & (3012,1032,+) \\
(1032,3120,+) & (3120,3012,-) & (3120,2301,+) \\
(1032,3201,+) & (3201,3120,-) & (3201,0123,+) \\
(1230,0123,+) & (0123,0213,-) & (0123,3012,+) \\
(1230,0213,+) & (0213,0123,-) & (0213,3201,+) \\
(1230,2301,-) & (2301,3012,-) & (2301,3201,+) \\
(1230,3201,-) & (3201,0213,-) & (3201,2301,+) \\
(2301,0213,+) & (0213,1230,-) & (0213,1032,+) \\
(2301,1230,+) & (1230,3201,-) & (1230,0123,+) \\
(2301,3012,-) & (3012,0123,-) & (3012,3120,+) \\
(2301,3120,-) & (3120,1032,-) & (3120,3201,+) \\
(2301,3201,-) & (3201,2301,-) & (3201,3012,+) \\
(2301,3201,+) & (3201,0213,-) & (3201,2301,+) \\
(2310,0123,-) & (0123,3201,-) & (0123,0213,+) \\
(2310,0213,-) & (0213,3012,-) & (0213,0123,+) \\
(2310,1032,+) & (1032,3012,-) & (1032,3201,+) \\
(2310,3012,+) & (3012,1032,-) & (3012,0213,+) \\
(3012,0123,-) & (0123,1230,-) & (0123,2310,+) \\
(3012,0213,+) & (0213,2301,-) & (0213,2310,+) \\
(3012,1032,-) & (1032,3012,-) & (1032,3201,+) \\
(3012,1032,+) & (1032,0213,-) & (1032,3012,+) \\
(3012,2301,+) & (2301,3120,-) & (2301,1230,+) \\
(3012,2310,-) & (2310,0213,-) & (2310,1032,+) \\
(3012,3120,+) & (3120,1032,-) & (3120,3201,+) \\
(3012,3201,-) & (3201,3120,-) & (3201,0123,+) \\
(3120,1032,-) & (1032,0213,-) & (1032,3012,+) \\
(3120,2301,+) & (2301,3201,-) & (2301,0213,+) \\
(3120,3012,-) & (3012,3201,-) & (3012,1032,+) \\
(3120,3201,+) & (3201,2301,-) & (3201,3012,+) \\
(3201,0123,+) & (0123,1230,-) & (0123,2310,+) \\
(3201,0213,-) & (0213,1230,-) & (0213,1032,+) \\
(3201,1032,-) & (1032,2310,-) & (1032,3120,+) \\
(3201,1230,+) & (1230,2301,-) & (1230,0213,+) \\
(3201,2301,-) & (2301,3201,-) & (2301,0213,+) \\
(3201,2301,+) & (2301,3012,-) & (2301,3201,+) \\
(3201,3012,+) & (3012,0123,-) & (3012,3120,+) \\
(3201,3120,-) & (3120,3012,-) & (3120,2301,+) \\
\end{array}
\end{displaymath}

\end{document}